\documentclass[10pt]{amsart}  
\usepackage{amsmath,amssymb,amsfonts,hyperref}

\title{Lecture notes: Semidefinite programs and harmonic analysis}

\author{Frank Vallentin} 

\address{F.~Vallentin, Centrum voor Wiskunde en Informatica (CWI),
Kruislaan 413, 1098 SJ Amsterdam, The Netherlands}

\email{f.vallentin@cwi.nl}

\thanks{The author was partially supported by
  the Deutsche Forschungsgemeinschaft (DFG) under grant SCHU
  1503/4.}



\date{September 11, 2008}

\newcommand{\defi}[1]{{\em #1}}

\newcommand{\R}{\mathbb{R}}

\newcommand{\Z}{\mathbb{Z}}
\newcommand{\C}{\mathbb{C}}

\newcommand{\MC}{\mathcal{C}}

\newtheorem{defin}{Definition}[section]
\newtheorem{definition}[defin]{Definition}
\newtheorem{proposition}[defin]{Proposition}
\newtheorem{theorem}[defin]{Theorem}

\newtheorem{lemma}[defin]{Lemma}
\newtheorem{example}[defin]{Example}
\newtheorem{exercise}[defin]{Exercise}

\DeclareMathOperator{\SO}{SO}

\DeclareMathOperator{\Aut}{Aut}
\DeclareMathOperator{\trace}{trace}
\DeclareMathOperator{\lin}{lin}

\newcommand{\Perm}{\operatorname{S}}
\newcommand{\Sn}{S^{n-1}} 
\newcommand{\Snm}{S^{n-2}} 
\newcommand{\Pol}{\operatorname{Pol}}

\newcommand{\On}{{\operatorname{O}(\R^n)}}
\newcommand{\Otwo}{{\operatorname{O}(\R^2)}}
\newcommand{\Orth}{\operatorname{O}}

\newcommand{\diag}{\operatorname{diag}}

\newcommand{\Harm}{\operatorname{Harm}}

\newcommand{\Stab}{\operatorname{Stab}}

\newcommand{\prodeucl}[2]{#1 \cdot #2}
\newcommand{\prodhaar}[2]{(#1,#2)}
\newcommand{\prodtrace}[2]{\langle #1, #2 \rangle}

\begin{document}

\begin{abstract}
  Lecture notes for the tutorial at the workshop \textsl{HPOPT 2008
    --- 10th International Workshop on High Performance Optimization
    Techniques (Algebraic Structure in Semidefinite Programming)}, June
  11th to 13th, 2008, Tilburg University, The Netherlands.
\end{abstract}

\maketitle

\tableofcontents

\section{Introduction}
\label{sec:introduction}

Semidefinite programming is a vast extension of linear programming and
has a wide range of applications: combinatorial optimization and
control theory are the most famous ones. Although semidefinite
programming has an enormous expressive power in formulating convex
optimization problems it has a few practical drawbacks: Robust and
efficient solvers, unlike their counterparts for solving linear
programs, are currently not available. So it is crucial to exploit the
problems' structure to be able to perform computations.

\medskip

In the last years many results in the area of semidefinite programming
were obtained for problems which have symmetry. The underlying
principle which was used here is the following: One simplifies the
original semidefinite program which is invariant under a group action
by applying an algebra isomorphism that maps a ``large'' matrix
algebra to a ``small'' matrix algebra. Then it is sufficient to solve
the semidefinite program using the smaller matrices.

\medskip

The aim of the tutorial is to give a general and explicit procedure to
simplify semidefinite programs which are invariant under the action of
a group.

\medskip

A \textit{(complex) semidefinite program} is an optimization problem
of the form
\begin{equation}
\label{sdp standard}
\max\{\prodtrace{C}{K} : \text{$\prodtrace{A_i}{K} = b_i$, $i =
1, \ldots, n$, and $K \succeq 0$}\},
\end{equation}
where $A_i \in \C^{V \times V}$, and $C \in \C^{V \times V}$ are
Hermitian matrices whose rows and columns are indexed by a finite set
$V$, $(b_1, \ldots, b_n)^t \in \R^n$ is a given vector and $K \in
\C^{V \times V}$ is a variable Hermitian matrix and where ``$K \succeq
0$'' means that $K$ is positive semidefinite. Here $\prodtrace{C}{K} =
\trace(CK)$ denotes the trace inner product between matrices. In the
following we denote the matrix entry $(x,y)$ of $K$ by $K(x,y)$
instead of the more familiar notion $K_{xy}$. This notation will make
our treatment of infinite matrices more natural.

\medskip

The punch line of the lecture is the following: Suppose that the
semidefinite program \eqref{sdp standard} is invariant under a group
$\Gamma$ of permutations of $V$: If $K$ is feasible for \eqref{sdp
  standard} then also $K_u$, defined by $K_u(x,y) = K(ux,uy)$, is
feasible for all $u \in \Gamma$, and the objective values of $K$ and
$K_u$ coincide. It is crucial to observe that to solve \eqref{sdp
  standard} it suffices to consider only those $K$ which satisfy
$K(x,y) = K(ux,uy)$ for all $u \in \Gamma$ because one can
\defi{symmetrize} every optimal solution $K$ to obtain a
$\Gamma$-invariant one:
\begin{equation*}
  \frac{1}{|\Gamma|}\sum_{u \in \Gamma} K(ux,uy).
\end{equation*}

\medskip

Now one can apply Bochner's characterization of $\Gamma$-invariant
positive semidefinite matrices. This is a classical result in harmonic
analysis. It says that one can represent any $\Gamma$-invariant $K$ by
a block-diagonal matrix $\diag(F_1, \ldots, F_l)$. One has the
representation
\begin{equation*}
K(x,y) = \sum_{k=1}^l \langle F_k, Z_k(x,y) \rangle,
\end{equation*}
where $F_k$ are positive semidefinite matrices of size $m_k \times
m_k$ and $Z_k(x,y) \in \C^{m_k \times m_k}$ are fixed basis matrices
(At this point the notation $Z_k(x,y)$ might be confusing, it
specifies a matrix and not a matrix entry, so it should be better
something like $Z_k^{(x,y)}$ but this looks awkward.). Here the
parameters $l$ and $m_1, \ldots, m_l$ as well as the basis matrices
$Z_k$ depend on the group $\Gamma$. The original semidefinite
program~\eqref{sdp standard} simplifies to
\begin{equation*}
\begin{split}
\max\Big\{\left\langle C, \sum_{k=1}^l \langle F_k, Z_k\rangle \right\rangle \;\; : \;\; & \left\langle A_i, \sum_{k=1}^l \langle F_k, Z_k\rangle \right\rangle = b_i,\; i = 1, \ldots, n,\\
& F_1, \ldots, F_l \succeq 0 \Big\},
\end{split}
\end{equation*}
and in some cases this simplification results into a huge saving.  The
advantage is that instead of dealing with matrices of size $|V| \times
|V|$ one has to deal with block diagonal matrices with $l$ block
matrices of size $m_1, \ldots, m_l$, respectively. In many
applications the sum $m_1 + \cdots + m_l$ is much smaller than $|V|$
and in particular many practical solvers take advantage of the block
structure to speed up the numerical calculations.

\medskip

In this lecture we develop the theory of ``block-diagonalization'' of
semidefinite programs starting from basic principles. We illustrate
this symmetry reduction for one specific semidefinite program: The
theta function for distance graphs on compact metric spaces. This is a
generalization of the Lov\'asz theta function for finite graphs. To
focus on this example has the following benefits: It shows the
geometric core of the process of symmetry reduction. In this context
it is natural to connect semidefinite programs with harmonic analysis
and the theory of group representation. In particular this provides
the possibility to consider infinite-dimensional semidefinite programs
which in many cases is a very convenient framework having potential
for future research.

\medskip

Then, this lecture gives the background for the recent development of
semidefinite programming bounds for combinatorial and geometric
packing problems initiated by Schrijver~\cite{Schrijver2}, and further
developed by Laurent~\cite{Laurent1}, Gijswijt, Schrijver,
Tanaka~\cite{GijswijtSchrijverTanaka}, Gijswijt~\cite{Gijswijt},
Bachoc, Vallentin~\cite{BachocVallentin1}, \cite{BachocVallentin2},
\cite{BachocVallentin3}, \cite{BachocVallentin4},
Musin~\cite{Musin2}. It also gives the background the developments
dealing with the more complicated case of noncompact metric spaces,
see Cohn, Elkies~\cite{CohnElkies}, de Oliveira Filho,
Vallentin~\cite{OliveiraVallentin}

\medskip

We want to stress (once more): 

\medskip

\begin{center}
\framebox{
\begin{minipage}{10cm}
  The techniques we present here apply to general semidefinite
  programs which are invariant under a symmetry group.
\end{minipage}
}
\end{center}

\bigskip

Then, one important remark: \textsl{For comments and suggestions concerning
these notes the author would be very grateful.}

\section{The theta function for distance graphs}
\label{sec:theta}

In this section we study the theta function of distance graphs in
compact metric spaces. These graphs can have infinitely many
vertices. The theta function is a generalization of the Lov\'asz theta
function for finite graphs, which originally was introduced by
Lov\'asz in the celebrated paper \cite{Lovasz1}. The Lov\'asz theta
function gives an upper bound for the stability number which one can
efficiently compute using semidefinite programming. The generalization
was studied by Bachoc, Nebe, de Oliveira Filho,
Vallentin~\cite{BachocNebeOliveiraVallentin}. The generalization also
gives an upper bound for the stability number which one can compute
using semidefinite programming. The main difference to the Lov\'asz
theta function is that one has to solve an infinite-dimensional
semidefinite program if the graph has infinitely many
vertices. However, we show that if the distance graph is symmetric,
solving this infinite-dimensional semidefinite program is feasible.

\medskip

Outline of this section:

\medskip

In Section \ref{ssec:distancegraphs} we provide the necessary
definitions from graph theory (stability number and distance graphs).
We show that finite graphs are distance graphs and we present the two
main examples: distance graphs on the Hamming cube (related to
error-correcting codes) and distance graphs on the unit sphere
(related to discrete geometry).

In Section~\ref{ssec:originalformulation} we discuss two possible
formulations of the Lov\'asz theta function.

Before we can formulate the generalization of the theta function in
Section~\ref{ssec:generalization} we review an infinite-dimensional
generalization of positive semidefinite matrices, so-called positive
Hilbert-Schmidt kernels, in Section~\ref{ssec:kernels}. There we in
particular focus on the spectral decomposition theorem which turns out
to be central in the following.

In Section~\ref{ssec:symmetry} we show how to exploit the symmetry of
distance graphs in order to simplify the computation of the theta
function. For this we introduce the automorphism group of a distance
graph and explain the essential tool of group invariant integration.

\subsection{Distance graphs}
\label{ssec:distancegraphs}

We start with some basic definitions from graph theory.  Let $G = (V,
E)$ be an \defi{undirected graph} given by a finite set $V$ of
\defi{vertices} and a subset $E \subseteq \binom{V}{2}$ of two-element
subsets of $V$ called \defi{edges}. Two vertices $x,y$ with $\{x,y\}
\in E$ are called \defi{adjacent} and a family of vertices
$(x_1,\ldots, x_n)$ in which every two consequent elements are
connected is called a \defi{path of length $n-1$}.

A \defi{stable set} (sometimes also called \defi{independent set}) of
a graph $G$ is a finite subset of the vertex-set in which no two
vertices are adjacent. The \defi{stability number} (or
\defi{independence number}) of a graph is the maximum cardinality of a
stable set of $G$:
\begin{equation*}
  \alpha(G) = \max\{|C| : C \subseteq V,\; 
  \text{$\{x,y\} \not\in E$ for all $x, y \in C$} \}.
\end{equation*}

If the graph has infinitely many vertices it may happen that there is
no maximum: The graph with vertex-set $V = S^{n-1} = \{x \in \R^n : x
\cdot x = 1\}$, in which two points are adjacent whenever they are
orthogonal, has stable sets of arbitrary cardinality. In this case, it
makes sense to replace the maximal cardinality of a finite stable set
by the maximal measure of a measurable stable set. Then in case of the
circle $S^1$ the stability number equals $\pi$. This approach has been
worked out by Bachoc, Nebe, de Oliveira Filho,
Vallentin~\cite{BachocNebeOliveiraVallentin}, but here we will stick
to the case when $\alpha(G)$ is finite.

\medskip

For the definition of a distance graph we use the triple
$(V,d,\mu)$. Here $V$ is a metric space with the distance function $d
: V \times V \to \R$ which is equipped with the finite, regular Borel
measure $\mu$. We assume that $V$ is separable and compact. Usually it
does not cause confusion if one uses only $V$ to specify the triple
$(V,d,\mu)$ and we refer to $V$ simply as a \defi{compact metric
  space}. The adjacency relation of a distance graph only depends on
the distance map $d$:

\begin{definition}
  Let $V$ be a compact metric space and let $I$ be a subset of the
  distances which may occur among distinct points in $V$. The
  \defi{distance graph} $G(V,I)$ is the graph with vertex-set $V$ and
  in which two vertices are adjacent if their distance lies in $I$. So
  the edge-set is $E = \big\{\{x,y\} : d(x,y) \in I\big\}$.
\end{definition}

The first example shows that finite graphs are distance graphs.

\begin{example} 
\label{ex:shortestpath}
Let $G = (V, E)$ be a finite graph.  By $d : V \times V \to \Z_{\geq
  0} \cup \{\infty\}$ we denote the length of a shortest path
connecting two vertices $x$ and $y$ in $G$ where we set $d(x, y) =
\infty$ whenever there is no connection at all. If the graph is
connected, that is if $d$ does not attain the value $\infty$, then $d$
defines a metric on $V$. The set $V$ comes with the uniform measure
$\mu$; we have $\mu(A) = |A|$ for any $A \subseteq V$.  The graph $G$
is a distance graph. Two vertices are adjacent whenever their distance
is equal to one. Hence, the adjacency relation only depends on the
distance map.
\end{example}

The second example comes from engineering, error correcting codes. In
fact, it can be viewed as a special case of the first example. Error
correcting codes have great practical importance for communication
across noisy channels and for storing and retrieving information on
media. The idea is that the sender adds redundant data to its messages
which allows the receiver to detect and correct errors under the
assumption that there is not unreasonable amount of noise. In
particular there is no need to ask the sender for additional data. The
book MacWilliams, Sloane~\cite{MacWilliamsSloane} is the
definitive reference in algebraic coding theory.

\begin{example}
\label{ex:hamming}
  Let $\{0,1\}^n$ be the set of binary strings of length $n$. It
  is also called the $n$-dimensional \defi{Hamming cube}. The Hamming
  cube is a metric space. The distance between two binary strings is
  measured by the \defi{Hamming distance} $\delta$, which is the
  number of entries in which they differ:
\begin{equation*}
\delta((x_1,\ldots,x_n),(y_1,\ldots,y_n)) = |\{i \in \{1,\ldots, n\} : x_i \neq y_i\}|.
\end{equation*}
The \defi{Hamming norm} (or \defi{Hamming weight}) of a vector $x \in
\{0,1\}^n$ is $\|x\| = \delta(x,0)$.
\end{example}

A central parameter is $A(n,d)$, the maximal cardinality of a subset
$C \subseteq \{0,1\}^n$ such that every two points in it have Hamming
distance at least $d$. In other words, $A(n,d)$ is the stability
number of the distance graph $G(\{0,1\}^n,\{1,\ldots,d-1\})$. A stable
set in this graph can correct $\lfloor \frac{d-1}{2} \rfloor$
errors. This value has been determined for various parameters of $n$
and $d$, but in general it is unknown to large
extend. Brouwer~\cite{Brouwer} maintains a list of lower and upper
bounds for $A(n,d)$ for $n \leq 28$.

\medskip

The third example comes from discrete geometry.

\begin{example}
\label{ex:sphere}
We consider the infinite graph $G(S^{n-1},(0,\theta))$ whose
vertex-set consists of all the points on the $(n-1)$-dimensional unit
sphere $S^{n-1} = \{x \in \R^n : x \cdot x = 1\}$. The distance
between two points on the unit sphere can be measured by the spherical
distance $d$. For $x,y \in S^{n-1}$ we have $d(x,y) = \arccos(x \cdot
y)$, where $x \cdot y$ denotes the Euclidean inner product. E.g.\
antipodal points have spherical distance $d(x,-x) = \pi$. The unit
sphere comes with a measure, the surface area $\omega$ which is
induced by the Lesbegue measure on $\R^n$. We have $\omega(S^{n-1})
= \frac{2\pi^{n/2}}{\Gamma(n/2)}$.
\end{example}

In the distance graph $G(S^{n-1},(0,\theta))$ each vertex $x$ is
adjacent to a spherical cap centered at $x$.  Stable sets of
$G(S^{n-1},(0,\theta))$ are called \defi{spherical codes with minimal
  angular distance $\theta$}. They are of special interest in
information theory. The stability number of this graph is also denoted
by $A(n,\theta)$.

The kissing number problem is equivalent to the problem of finding
$A(n,\pi/3)$. In geometry, the \defi{kissing number problem} asks for
the maximum number $\tau_n$ of unit spheres that can simultaneously
touch the unit sphere in $n$-dimensional Euclidean space without
pairwise overlapping. The touching points form a spherical code with
minimal angular distance $\pi/3$.  The value of $\tau_n$ is only known
for $n=1,2,3,4,8,24$. While its determination for $n=1,2$ is trivial,
it is not the case for other values of $n$.  The case $n=3$ was the
object of a famous discussion between Isaac Newton and David Gregory
in 1694. For a historical perspective of this discussion we refer to
Casselman~\cite{Casselman}. The first valid proof of the fact
``$\tau_3=12$'', like in the icosahedron configuration, was only given
in 1953 by Sch\"utte and van der
Waerden~\cite{SchuetteVanDerWaerden}. Odlyzko and
Sloane~\cite{OdlyzkoSloane}, and independently
Levenshtein~\cite{Levenshtein}, proved $\tau_8=240$ and
$\tau_{24}=196560$ which are respectively the number of shortest
vectors in the root lattice $E_8$ and in the Leech lattice.  In 2003,
Musin~\cite{Musin} succeeded to prove the conjectured value
$\tau_4=24$, which is the number of shortest vectors in the root
lattice $D_4$. See also the survey Pfender,
Ziegler~\cite{PfenderZiegler}. The known lower and upper bounds for
$\tau_n$, with $n \leq 10$, are given in Bachoc,
Vallentin~\cite{BachocVallentin1}.

\subsection{Original formulations for theta}
\label{ssec:originalformulation}

In the original paper \cite{Lovasz1}, the Lov\'asz theta
function was given as the solution of a semidefinite program which
involves a positive semidefinite matrix $K \in \R^{V \times V}$ whose
rows and columns are indexed by the finite vertex-set of a graph $G =
(V, E)$. In \cite[Theorem 3, Theorem 4]{Lovasz1} Lov\'asz gave the
following two formulations of the theta function, which are equivalent
by semidefinite programming duality:
\begin{equation}
\label{eq:origthetaprimal}
\begin{split}
\vartheta(G) =  \max\big\{\sum_{x \in V} \sum_{y \in V} K(x,y) \;\; : \;\;& \text{$K \in \R^{V \times V}$ is positive semidefinite},\\[-0.25cm]
& \text{$\sum_{x \in V} K(x,x) = 1$},\\
& \text{$K(x,y) = 0$ if $\{x,y\} \in E$}\big\},
\end{split}
\end{equation}
and
\begin{equation}
\label{eq:origthetadual}
\begin{split}
\vartheta(G) = \min\big\{ \lambda \;\; : \;\; & \text{$K \in \R^{V
\times V}$ is positive semidefinite},\\ 
& \text{$K(x,x) = \lambda - 1$ for all $x \in V$,}\\ 
& \text{$K(x,y) = -1$ if $\{x,y\} \not\in E$}\big\}.
\end{split}
\end{equation}

There are many equivalent definitions of the theta function. Possible
alternatives are reviewed by Knuth~\cite{Knuth}. One can rephrase
\eqref{eq:origthetadual} by saying that $\lambda$ is the minimum of
the largest eigenvalue of any symmetric matrix $K \in \R^{V \times V}$
such that $K(x,y) = 1$ whenever $x = y$ or if $x$ and $y$ are not
adjacent.

\begin{exercise}
  Prove this statement.
\end{exercise}

It is easy to see that the theta function gives an upper bound for the
stability number of a finite graph: Let $C \subseteq V$ be a stable
set of maximal cardinality. Consider the column vector $1_C \in \R^V$
which is the characteristic vector of $C$ defined by $1_C(x) = 1$ if
$x \in C$ and $1_C(x) = 0$ otherwise. Then, the rank-$1$ matrix $K =
1/|C| 1_C (1_C)^t$, which componentwise is $K(x,y) = 1/|C| 1_C(x) 1_C(y)$, is
feasible for \eqref{eq:origthetaprimal}. Thus, $\vartheta(G) \geq
\alpha(G)$.

One can strengthen the equalities in \eqref{eq:origthetaprimal} and
\eqref{eq:origthetadual} by the inequalities $K(x,y) \geq 0$ and
$K(x,y) \leq -1$ respectively.  This strengthening was introduced by
Schrijver~\cite{Schrijver1}. Sometimes it is denoted by
$\vartheta'(G)$ and we have $\vartheta(G) \geq \vartheta'(G) \geq
\alpha(G)$.

Using semidefinite programming one can compute the theta function in
polynomial time, in the sense that one can approximate it with any
given precision. In general the upper bound given by the theta
function is weak: Feige~\cite{Feige} proved that there exists a
constant $c$ and an infinite family of graphs on $n$ vertices for
which $\vartheta(G)/\alpha(G) > n/2^{c\sqrt{\log n}}$ holds. So in
particular it does not give a $n^{1-\varepsilon}$-approximation for
any fixed $\varepsilon > 0$. It is weak for a good reason:
H\aa{}stad~\cite{Hastad} showed that for any fixed $\varepsilon > 0$
one cannot approximate the stability number of a general graph with
$n$ vertices within a factor of $n^{1 - \varepsilon}$ in polynomial
time unless any problem in $\mathrm{NP}$ can be solved in expected,
probabilistic polynomial time. However, as we will see later, it is
sometimes surprisingly good, especially for symmetric graphs.

\subsection{Positive Hilbert-Schmidt kernels}
\label{ssec:kernels}

To be able to define the theta function for infinite graphs we need a
concept of positive semidefinite matrices with infintely many rows and
columns. We use positive Hilbert-Schmidt kernels for this, which are
well-studied objects in functional analysis. Many familiar facts about
positive semidefinite matrices can be generalized to positive
Hilbert-Schmidt kernels. Historically, the central results of the
theory of Hilbert-Schmidt kernels have been developed at the beginning
of the 20th century mainly by Fredholm, Hilbert, Mercer, Schmidt. The
classical text books Courant, Hilbert~\cite{CourantHilbert} and Riesz,
Sz.-Nagy~\cite{RieszNagy} are beautiful expositions. A modern
treatment is contained for example in the comprehensive books by Reed,
Simon~\cite{ReedSimon}.

\medskip

Let us list some basic properties of positive Hilbert-Schmidt
kernels. Here it is sometimes enlightening to compare positive
Hilbert-Schmidt kernels with the more familiar positive semidefinite
matrices.

\medskip

By $\MC(V)$ we denote the set of complex-valued continuous functions
$f : V \to \C$ which is an inner product space by
\begin{equation}
\label{eq:innerproduct}
(f,g) = \int_V f(x) \overline{g(x)} d\mu(x),
\end{equation}
and a normed space by $\|f\| = \sqrt{(f,f)}$.  We say that a family of
continuous functions $e_1, e_2, \ldots \in \MC(V)$ is an
\defi{orthonormal system} if
\begin{equation*}
\text{$(e_k,e_k) = 1$ and $(e_k, e_l) = 0$, whenever $k \neq l$.}
\end{equation*}
We say that it is \defi{complete} if every continuous function can be
approximated arbitrarily well by finite linear combinations in terms
of convergence in the mean, i.e.\ convergence with respect to the norm
$\|\cdot\|$ introduced above. Let $e_1, e_2, \ldots, e_d$ be an
orthonormal system, then we have the fundamental \defi{Bessel's
  inequality}
\begin{equation*}
0 \leq \|f - \sum_{k=1}^d (f,e_k) e_k\|^2 = \|f\|^2 - \sum_{k=1}^d |(f,e_k)|^2.
\end{equation*}

By $\MC(V \times V)$ we denote the set of continuous functions $K : V
\times V \to \C$.  The elements of $\MC(V \times V)$ are traditionally
called \defi{kernels} because they appear as integral kernels in the
theory of integral equations: For instance the \defi{homogeneous
  Fredholm integral equation of the second type} is an integral
equation of the form
\begin{equation}
\label{eq:fredholm}
\lambda f(x) - \int_V K(x,y) f(y) d\mu(y) = 0,
\end{equation}
where $K \in \MC(V \times V)$ is a given integral kernel, and $\lambda
\in \C$ and $f \in \MC(V)$ are to be determined. In fact, solutions to
\eqref{eq:fredholm} are given by \defi{eigenvalues} and
\defi{eigenfunctions} of the linear map
\begin{equation*}
T_K : \MC(V) \to \MC(V), \quad
T_K(f)(x) = \int_V K(x,y) f(y) d\mu(y),
\end{equation*}
i.e.\ nontrivial $f \in \MC(V)$ and $\lambda \in \C$ with $T_K(f) =
\lambda f$.

\medskip

In the following we only consider \defi{symmetric} or \defi{Hermitian}
kernels. They satisfy $K(x,y) = \overline{K(y,x)}$ for all $x, y \in
V$.  A kernel $K \in \MC(V \times V)$ is called \defi{positive} if for
any nonnegative integer $m$, any points $x_1, \ldots, x_m \in V$, and
any complex numbers $u_1, \ldots, u_m$, we have
\begin{equation*}
\sum_{i = 1}^m \sum_{j = 1}^m K(x_i,x_j)
u_i \overline{u_j} \geq 0.
\end{equation*}
So any finite matrix $(K(x_i,x_j))_{i,j}$ which one extracts from $K$
has to be positive semidefinite. We refer to positive, symmetric,
continuous kernels as \defi{positive Hilbert-Schmidt kernels} or
simply as \defi{positive kernels}.

\medskip

Let us turn to the spectral decomposition of positive Hilbert-Schmidt
kernels $K \in \MC(V \times V)$ where we first recall what happens for
positive semidefinite matrices $A \in \C^{n \times n}$ (or in a more
snobbish notation $A \in \MC(\{1,\ldots, n\} \times \{1,\ldots,
n\})$).

\medskip

Positive semidefinite matrices $A \in \C^{n \times n}$ can be
diagonalized: There is an orthonormal basis $e_1, \ldots, e_n$ of
$\C^n$ consisting of eigenvectors of $A$, all eigenvalues $\lambda_1,
\ldots, \lambda_n$ are nonnegative real numbers, and the $(x,y)$
entry of $A$ can be written as
\begin{equation*}
A(x,y) = \sum_{k = 1}^n \lambda_k e_k(x) \overline{e_k(y)}.
\end{equation*}
The \defi{rank} of $A$ is the number of nonzero eigenvalues, counted
with multiplicity.  Suppose that the eigenvalues are ordered in
descending order
\begin{equation*}
\lambda_1 \geq \lambda_2 \geq \ldots \geq \lambda_n \geq 0,
\end{equation*}
then they can be determined from the following max-min characterization:
\begin{equation*}
  \lambda_k = \max_{S_k} \min_{f \in S_k, (f,f) = 1}  (Af,f),
\end{equation*}
where $S_k$ runs through all $k$-dimensional subspaces of
$\C^n$. Geometrically, the eigenvectors form the principal axes of the
ellipsoid $\{f \in \C^n : (Af,f) = 1\}$ defined by $A$. The square
roots of the reciprocals of the corresponding eigenvalues give the
length of these axes.

\medskip

Essentially the same statements hold true for positive Hilbert-Schmidt
kernels. The main difference between positive semidefinite matrices
and positive kernels is that the latter have infinitely many,
different eigenvalues.

\medskip

Let $K \in \MC(V \times V)$ be a positive Hilbert-Schmidt kernel. To
get the flavour of the reasoning one needs when working with
Hilbert-Schmidt kernels we show using Bessel's inequality that the
multiplicity of every nonzero eigenvalue is finite: Let $e_1, \ldots,
e_d$ be an orthonormal system of eigenfunctions of a nonzero
eigenvalue $\lambda$. For fixed $x \in V$ we apply Bessel's inequality
to the function $y \mapsto K(x,y)$ and get
\begin{equation*}
\int_V |K(x,y)|^2 d\mu(y) \geq \sum_{k=1}^d \left|\int_V K(x,y) e_k(y) d\mu(y) \right|^2 = \sum_{k=1}^d |\lambda e_k(x)|^2.
\end{equation*}
Integrating both sides again yields an upper bound for $d$:
\begin{equation*}
\int_V\int_V |K(x,y)|^2 d\mu(y) d\mu(x) \geq |\lambda|^2 d.
\end{equation*}

Furthermore, one can show that the series of squared eigenvalues
converges to $0$. So there are two possibilities: $K$ has finitely
many or infinitely many nonzero eigenvalues. In the latter case, the
only accumulation point of the eigenvalues is $0$. The \defi{rank} of
a kernel is the number of nonzero eigenvalues counted with
multiplicity.

We have the \defi{spectral decomposition} of a positive
Hilbert-Schmidt kernel:
\begin{equation*}
K(x,y) = \sum_{k} \lambda_k e_k(x)\overline{e_k(y)}.
\end{equation*}
Here $\lambda_k$ are the eigenvalues of $K$, which are all
nonnegative, and the $e_k$ form a complete orthonormal system. In the
case of infinitely many different eigenvalues the right hand side
converges absolutely and uniformly to the left hand side. In the case
of finitely many different eigenvalues, the finite rank case, the
series degenerates to a finite sum. In the literature this statement
about the spectral decomposition of positive Hilbert-Schmidt kernels
is often called \defi{Mercer's theorem}.

The max-min characterization of the eigenvalues of $K$ reads as
follows: Suppose that the eigenvalues of $K$ are ordered in descending
order
\begin{equation*}
\lambda_1 \geq \lambda_2 \geq \ldots \geq 0,
\end{equation*}
then they can be determined from the following max-min
characterization:
\begin{equation*}
\lambda_k = \max_{S_k} \min_{f \in S_k, (f,f) = 1}  (T_K(f),f),
\end{equation*}
where $S_k$ runs through all $k$-dimensional subspace of $\MC(V)$. In
fact the max-min characterization of the eigenvalues is used in
Courant, Hilbert~\cite{CourantHilbert} and Riesz,
Sz.-Nagy~\cite{RieszNagy} to prove the spectral decomposition theorem.

\begin{exercise}
  Determine the eigenvalues of the inner product kernel $K : S^{n-1}
  \times S^{n-1} \to \C$ defined by $K(x,y) = x \cdot y$. Give the
  spectral decomposition of it in the case $n = 2$. (Maybe you want to
  come back to this exercise after Section~\ref{sec:spherical}.)
\end{exercise}

\subsection{The theta function for infinite graphs}
\label{ssec:generalization}

Now we are ready to generalize the theta function to possibly infinite
distance graphs $G(V,I)$.  It will turn out that for the combinatorial
and geometric packing problems we consider it is convenient to work
with the second formulation \eqref{eq:origthetadual}; possibilities to
work with the first formulation are considered in Bachoc, Nebe, de
Oliveira Filho, Vallentin~\cite{BachocNebeOliveiraVallentin}.  In the
second formulation we simply replace the condition that $K$ is a
positive semidefinite matrix by the condition that $K$ is a positive
kernel.

\begin{equation}
\label{eq:generalization}
\begin{split}
\vartheta(G(V,I)) = \min\big\{ \lambda \;\; : \;\; & \text{$K \in \MC(V \times V)$ is positive},\\ 
& \text{$K(x,x) = \lambda - 1$ for all $x \in V$,}\\ 
& \text{$K(x,y) = -1$ if $\{x,y\} \not\in E$\}}\big\}.
\end{split}
\end{equation}

When we are working with an infinite-dimensional semidefinite program
we can no longer use the simple duality argument given in
Section~\ref{ssec:originalformulation} to show that the theta function
gives an upper bound for the stability number. So we have to provide a
new (simple) one: Instead of constructing a feasible kernel from a
stable set and going to the dual semidefinite program, we show that
any feasible kernel can be used as a test function on any stable set
$C$ to bound the cardinality of $C$.

\begin{theorem}
\label{thm:upperbound}
Let $G(V,I)$ be a distance graph. Then, the theta function is an upper
bound for the stability number: $\vartheta(G(V,I)) \geq \alpha(G(V,I))$.
\end{theorem}

\begin{proof}
  Let $C$ be a stable set of $G(V,I)$ and let $K$ be a kernel which
  satisfies the conditions in \eqref{eq:generalization}. We consider
  the $|C| \times |C|$ matrix $(K(c,c'))_{(c,c') \in C^2}$ which we
  extract from $K$. Because this matrix is positive semidefinite we
  have
\begin{equation*}
0 \leq \sum_{(c,c') \in C^2} K(c,c').
\end{equation*}
On the other hand,
\begin{equation*}
\sum_{(c,c') \in C^2} K(c,c')= \sum_{c} K(c,c) + \sum_{c \neq c'} K(c,c')
\leq |C| K(c,c) - |C|(|C| - 1),
\end{equation*}
so that $|C| - 1 \leq K(c,c)$ and the statement follows.
\end{proof}

In the same way one can define $\vartheta'(G(V,I))$ by strengthening
the condition $K(x,y) = 1$ to $K(x,y) \leq -1$. Then, one has
\begin{equation*}
\vartheta(G(V,I)) \geq \vartheta'(G(V,I)) \geq \alpha(G(V,I)).
\end{equation*}

\subsection{Exploiting symmetry}
\label{ssec:symmetry}

\textsl{We should start to speak about symmetry.} In this section we
show the basic step to simplify the computation of the theta function
for symmetric distance graphs. Note that this basic step can be
applied to other semidefinite programs which are invariant under the
action of a group. Once one has identified this symmetry, this basic
step, as well as the following steps, can be applied to some extend
mechanically.

\medskip

Usually many symmetries of a distance graph come from the \defi{automorphism
  group} of the underlying metric space $V$. It is the group of all
permutations $u : V \to V$ which leave the distance map $d$ invariant:
\begin{equation*}
\Aut(V) = \{u : V \to V : \text{$d(x,y) = d(ux,uy)$ for all $x,y \in V$}\}.
\end{equation*}

Let us take a look at the automorphism group of the three examples
\ref{ex:shortestpath}--\ref{ex:sphere}:

\begin{example}
  The automorphism group of a metric space defined by the shortest
  paths in a connected graph $G = (V,E)$ equals the automorphism group
  of the underlying graph:
\begin{equation*}
\Aut(G) = \{u : V \to V : \text{$\{x,y\} \in E$ if and only if $\{ux,uy\} \in E$}\}.
\end{equation*}
\end{example}

Computing the automorphism group is not easy. Deciding whether the
automorphism group is trivial is as difficult as the graph isomorphism
problem. For this no polynomial time algorithm is known. The graph
isomorphism problem is generally believed to lie in $\text{NP} \cap
\text{co-NP}$. So it is unlikely that it is $\text{NP}$-hard. For more
information on the computational complexity of this problem we refer
to the book K\"obler, Sch\"oning,
T\'oran~\cite{KoeblerSchoeningToran}. On the practical side the
program \texttt{nauty} of McKay \cite{McKay} is a very useful tool for
computing the automorphism group of a graph.

\begin{example}
  The automorphism group of the Hamming cube $\{0,1\}^n$ has order $2^n n!$.
  It is generated by all $n!$ permutations of the $n$ coordinates and
  all $2^n$ switches $0 \leftrightarrow 1$ which one can identify with
  the Hamming cube $\{0,1\}^n$ itself where one considers addition modulo
  $2$.
\end{example}

\begin{example}
  The automorphism group of the unit sphere is the \defi{orthogonal
    group}. It is the group of orthogonal matrices
\begin{equation*}
\Aut(S^{n-1}) = \On = \{u \in \R^{n \times n}: u^t u = I_n\},
\end{equation*}
where $I_n$ is the $(n \times n)$-identity matrix.  
\end{example}

The orthogonal group is generated by a reflection over some hyperplane
and all rotations. It preserves the inner product (and thus the spherical
distance) as well as the Euclidean distance. Another way to view it is
to observe that the elements of $\On$ map orthonormal systems of
$\R^n$ to orthonormal systems.

\medskip

\begin{center}
\fbox{Crucial observation}
\end{center}

\medskip

Now we come to a crucial observation: In the computation of the theta
function \eqref{eq:generalization} of a distance graph $G(V,I)$ one
can restrict the semidefinite program to positive Hilbert-Schmidt
kernels which are invariant under the automorphism group of $V$. The
optimal objective value of this restricted semidefinite program remains
unchanged, and the matrix sizes which are needed in the practical
computation can be drastically smaller.

\medskip

We say that a positive kernel $K \in \MC(V \times V)$ is
\defi{$\Aut(V)$-invariant} if $K(ux,uy) = K(x,y)$ holds for all $u \in
\Aut(V)$ and all $x,y \in V$. For instance, the positive kernel $K \in
\MC(S^{n-1} \times S^{n-1})$ which is defined by the inner product
$K(x,y) = x \cdot y$ is $\On$-invariant.

\medskip

If the graph is finite, this crucial observation is easy. If $K$ is a
feasible solution of \eqref{eq:generalization} so is its
$\Aut(V)$-invariant \defi{group average}
\begin{equation}
\label{eq:groupaverage}
\tilde{K}(x,y) = \frac{1}{|\Aut(V)|}\sum_{u \in \Aut(V)} K(ux,uy).
\end{equation}
This is also easy to verify, e.g.\ we have for $x,y \in V$ with
$\{x,y\} \not\in E$
\begin{equation*}
  \tilde{K}(x,y) = \frac{1}{|\Aut(V)|}\sum_{u \in \Aut(V)} K(ux,uy) = \frac{|\Aut(V)|}{|\Aut(V)|} (-1) = -1.
\end{equation*}
The matrix $\tilde{K}$ is positive semidefinite because it is a
nonnegative sum of $|\Aut(V)|$ positive semidefinite matrices
$K_u(x,y) = K(ux,uy)$. Also the objective values of $K$ and $\tilde{K}$
coincide.

\medskip

If the distance graph is infinite we need a replacement for the finite
sum~\eqref{eq:groupaverage}. For this we use the invariant integral of
the group $\Aut(V)$. If $\Aut(V)$ is finite, then the invariant
integral is given by the sum~\eqref{eq:groupaverage}.  Generally, this
invariant integral is defined for compact topological groups and it
can be constructed from the Haar measure (see e.g.\ the original
Haar~\cite{Haar}, or the standard text on measure theory
Taylor~\cite{Taylor}). Lov\'asz \cite{Lovasz2} gives an elementary,
combinatorial construction of the invariant integral based on the
marriage theorem of matching theory.

\medskip

Here we only give the defining properties of the invariant integral. For
the explicit construction we refer the interested reader to the above
mentioned literature. Actually one could establish the invariant
integral in our three examples directly.

\medskip

Let $\Gamma$ be a compact topological group. A \defi{topological
  group} is a group which is equipped with a topology so that
multiplication and inversion are continuous functions. Examples of
compact topological groups are all finite groups and the orthogonal
group $\On$.  The topology here comes from the topology defined on $(n
\times n)$-matrices. A topological group which is not compact is the
group of translations $(\R^n,+)$.

\medskip

We consider the space of complex-valued continuous functions on the
group $\MC(\Gamma)$.  Then there is a unique map $\int_{\Gamma} :
\MC(\Gamma) \to \C$, the \defi{invariant integral}, with the following
properties:

(a) Linearity 
\begin{equation*}
  \int_{\Gamma} \alpha f + \beta g = \alpha \int_{\Gamma} f + \beta \int_{\Gamma} g \quad \text{for all $f,g \in \MC(\Gamma)$, $\alpha, \beta \in \C$.}
\end{equation*}

(b) Monotonicity
\begin{equation*}
\int_{\Gamma} |f| \geq 0 \quad \text{for all $f \in \MC(\Gamma)$.}
\end{equation*}

(c) Normalization
\begin{equation*}
\int_{\Gamma} 1 = 1.
\end{equation*}

(d) Invariance
\begin{equation*}
\int_{\Gamma} f(vu) = \int_{\Gamma} f(u) \quad \text{for all $v \in \Gamma$.}
\end{equation*}

\bigskip

Two quick examples: If $\Gamma$ is a finite group, the invariant
integral simply is
\begin{equation*}
\int_{\Gamma} f = \frac{1}{|\Gamma|}\sum_{u \in \Gamma} f(u).
\end{equation*}
For the two-dimensional rotation group $\SO(\R^2) = \{u \in \Otwo : \det
u = 1\}$ the invariant integral is
\begin{equation*}
\int_{\SO(\R^2)} f = \frac{1}{2\pi} \int_0^{2\pi} f\left(\begin{pmatrix} \cos \theta & -\sin \theta\\ \sin\theta & \cos\theta\end{pmatrix}\right) d\theta.
\end{equation*}

\medskip

Now the proper replacement of the group average given in
\eqref{eq:groupaverage} for general distance graphs $G(V,I)$ is
\begin{equation}
\label{eq:groupaverage2}
\tilde{K}(x,y) = \int_{\Aut(V)} K(ux, uy),
\end{equation}
where we take the invariant integral of the function $u \mapsto
K(ux,uy) \in \MC(\Gamma)$.

\section{Harmonic analysis}
\label{sec:harmonic}

In the last section we showed that in the computation of the theta
function it suffices to consider invariant kernels only.  If one
chooses an appropriate basis of the invariant kernels, then it is
possible to express positivity in an efficient way, namely as the
positive semidefiniteness of a couple of smaller block matrices. To
make this statement precise we have to use some harmonic analysis
which we develop in this section starting from basic principles.

\medskip

To find an appropriate basis we use basic tools from harmonic analysis
and representation theory of compact groups which are usually
associated with the Peter-Weyl theorem. In Section~\ref{ssec:basics}
we present the basic notions and in Section~\ref{ssec:peterweyl} we
state and prove the Peter-Weyl theorem. In Section~\ref{ssec:bochner}
we apply it to find Bochner's characterization of positive, invariant
kernels. Throughout this section we develop the theory with the help
of the spectral theory of positive kernels. Then, in the following
sections we show how the Peter-Weyl theorem and Bochner's
characterization specializes for the computation of the theta
functions of distance graphs on the Hamming cube and on the unit
sphere.

\medskip

Literature: The literature on harmonic analysis and group
representations is huge. Probably the book by Serre~\cite{Serre} on the
linear representation theory of finite groups is the most
prominent. Other books which emphasize the connection of group
representations of finite groups to probability is Diaconis
\cite{Diaconis}, one emphasizing the connection to algebraic
combinatorics is Sagan \cite{Sagan}, and one connecting to chemistry,
error-correcting codes, data analysis, graph theory, and probability is
Terras \cite{Terras}. Further good sources are: Br\"ocker and tom
Dieck \cite{BroeckerTomDieck}, Vinberg \cite{Vinberg}, Carter, Segal,
Macdonald \cite{CarterSegalMacdonald}, Goodman, Wallach
\cite{GoodmanWallach}. The articles Borel \cite{Borel}, Gross
\cite{Gross} and Slodowy \cite{Slodowy} illuminate the historical
background.

\subsection{Basic tools}
\label{ssec:basics}

Let us recall the set-up: $V$ is a compact metric space with measure
$\mu$, and $\Gamma = \Aut(V)$ is the automorphism group of $V$. We
define the action of $\Gamma$ on the linear space of complex-valued
continuous functions $\MC(V)$ by
\begin{equation*}
uf(x) = f(u^{-1}x).
\end{equation*}
We assume that the measure $\mu$ is $\Gamma$-invariant ($\mu(uA) =
\mu(A)$), so that we have an \defi{invariant inner product}
\eqref{eq:innerproduct} on the space of continuous functions:
\begin{equation*}
  \text{$(f,g) = (uf, ug)$ for all $f,g \in \MC(V)$ and $u \in \Gamma$.}
\end{equation*}
The action of $\Gamma$ on the vector space $\MC(V)$ is linear, that is
we have
\begin{equation*}
\text{$u(\alpha f + \beta g) = \alpha u f + \beta u g$ for all $f,g \in \MC(V)$, $\alpha, \beta \in \C$ and $u \in \Gamma$.}
\end{equation*}
Linear actions are called \defi{group representations}. More
vocabulary: A subspace $S \subseteq \MC(V)$ is called
\defi{$\Gamma$-invariant} if $uS = S$ for all $u \in \Gamma$, i.e.\ if
for every $u \in \Gamma$ and for every $f \in S$ we have $uf \in S$ as
well. A nonzero subspace $S$ is called \defi{$\Gamma$-irreducible} if
$\{0\}$ and $S$ are the only $\Gamma$-invariant subspaces of $S$. Let
$S$ and $S'$ be two invariant subspaces. A linear map $T : S \to S'$
is called a \defi{$\Gamma$-map} if $T(uf) = uT(f)$ for all $u \in
\Gamma$, and $f \in \MC(V)$. We say that $S$ and $S'$ are
$\Gamma$-equivalent if there is a bijective $\Gamma$-map between
them. If it is clear from the context which group we consider, we
frequently omit the prefix ``$\Gamma$-''.

\medskip

Two little lemmas will make our live easier. 

\medskip

The first one is also called \defi{Maschke's theorem}.

\begin{lemma}
  Let $S$ be a invariant subspace and let $U \subseteq S$ be an
  invariant subspace of $S$. Then its orthogonal complement in $S$,
  given by
\begin{equation*}
  U^{\perp} = \{g \in S : \text{$(f,g) = 0$ for all $f \in
    U$}\}
\end{equation*}
is invariant as well.
\end{lemma}

\begin{proof}
  This comes from the invariant inner product: We have for all $u \in
  \Gamma$
\begin{equation*}
  \text{$(f,ug) = (u^{-1}f, g) = 0$ whenever $f \in U$, $g \in U^{\perp}$.}
\qedhere
\end{equation*}
\end{proof}

Maschke's theorem implies that a finite-dimensional invariant subspace
which is not irreducible splits into an orthogonal sum of irreducible
subspaces.

\medskip

The second lemma is \defi{Schur's lemma}.

\begin{lemma}
  Let $S$ and $S'$ be two irreducible subspaces, and let $T : S \to
  S'$ be a $\Gamma$-map. If $S$ and $S'$ are not equivalent, then $T =
  0$. If they are equivalent then either $T = 0$ or $T$ is bijective.
\end{lemma}

\begin{proof}
  The kernel of $T$ is an invariant subspace of $S$, and the image is
  an invariant subspace of $S'$. Since $S$ and $S'$ are irreducible
  subspaces we are left with four possibilities:
\begin{enumerate}
\item The kernel of $T$ is $\{0\}$ and the image of $T$ is
  $\{0\}$. This could only happen when $S = \{0\}$.
\item The kernel of $T$ is $\{0\}$ and the image of $T$ is $S'$. Then,
  $T$ is a bijective $\Gamma$-map and $S$ and $S'$ are equivalent
  subspaces.
\item The kernel of $T$ is $S$ and the image of $T$ is $\{0\}$. Then,
  $T$ must be the zero map.
\item The kernel of $T$ is $S$ and the image of $T$ is $S'$. This
  could only happen when $S' = \{0\}$.\qedhere
\end{enumerate}
\end{proof}

We apply Schur's lemma to establish some orthogonality relations which
will be essential later.

\begin{lemma}
\label{lem:orthogonalityrelation}
Let $S$ and $S'$ be two irreducible subspaces. Let $e_1, \ldots, e_h$
be a complete orthonormal system of $S$ and let $e'_1, \ldots,
e'_{h'}$ be one of $S'$.
\begin{enumerate}
\item If $S$ and $S'$ are not equivalent, then they are orthogonal to
  each other:
\begin{equation*}
\text{$(e_i, e'_j) = 0$, for $i = 1, \ldots, h$, $j = 1, \ldots, h'$}
\end{equation*}
\item If $S$ and $S'$ are equivalent and if $T : S \to S'$ is a
  bijective $\Gamma$-map, mapping $e_i$ to $e'_i$, then there is a
  constant $c$ so that
\begin{equation*}
(e_i,e'_j) = \left\{
\begin{array}{ll}
c & \text{if $i = j$},\\
0 & \text{otherwise.}
\end{array}
\right.
\end{equation*}
\end{enumerate}
\end{lemma}

\begin{proof}
Define the linear map $A : S \to S'$ by
\begin{equation}
\label{eq:A}
A(e_i) = \sum_{j = 1}^{h'} (e'_j, e_i) e'_j.
\end{equation}
This is a $\Gamma$-map. Now the first claim follows immediately from
Schur's lemma: We have $A = 0$ and so $(e'_j, e_i) = 0$. For the
second claim we assume $A \neq 0$. Then $A$ is bijective by Schur's
lemma. Consider the endomorphism $T^{-1}A : S \to S$, $T^{-1}A(e_i) =
\sum_{j=1}^h (e'_j, e_i) e_j$. Since we work over the complex numbers
and $T^{-1}A$ is bijective, it has a nonzero eigenvalue $c$. The
corresponding eigenspace is an invariant subspace of $S$, so it has to
be equal to $S$. Hence $T^{-1}A$ is $c$ times the identity map, which
proves the second claim.
\end{proof}

\begin{exercise}
  Check that the map $A$ defined in \eqref{eq:A} is indeed a
  $\Gamma$-map.
\end{exercise}

\subsection{The Peter-Weyl theorem}
\label{ssec:peterweyl}

The theorem of Peter and Weyl \cite{PeterWeyl} and \cite{Weyl} is the
starting point of harmonic analysis of compact groups. It connect
Fourier analysis with group representations. It shows that the space
$\MC(V)$ decomposes orthogonally into finite-dimensional irreducible
subspaces and that the space $\MC(V)$ has a complete orthonormal
system which is ``in harmony'' with the group $\Gamma = \Aut(V)$.

\begin{theorem}
\label{thm:peterweyl}
  All irreducible subspaces of $\MC(V)$ are of finite dimension. The
  space $\MC(V)$ decomposes orthogonally as
\begin{equation*}
\MC(V) = \bigoplus_{k = 0, 1, \ldots} H_k, 
\end{equation*}
and the space $H_k$ decomposes orthogonally as
\begin{equation*}
H_k = \bigoplus_{i = 1, 2, \ldots, m_k} H_{k,i},
\end{equation*}
where $H_{k,i}$ is irreducible, and $H_{k,i}$ is equivalent to
$H_{k',i'}$ if and only if $k = k'$. The dimension $h_k$ of $H_{k,i}$
is finite, \emph{but the multiplicity $m_k$ can potentially be infinite}.

In other words, $\MC(V)$ has a complete orthonormal system $e_{k,i,l}$,
where $k = 0, 1, \ldots$, $i = 1, 2, \ldots, m_k$, $l = 1, \ldots,
h_k$ so that
\begin{enumerate}
\item the space $H_{k,i}$ spanned by $e_{k,i,1}, \ldots, e_{k,i,h_k}$
  is irreducible,
\item the spaces $H_{k,i}$ and $H_{k',i'}$ are equivalent if and only
  if $k = k'$,
\item there are $\Gamma$-maps $\phi_{k,i} : H_{k,1} \to H_{k,i}$
  mapping $e_{k,1,l}$ to $e_{k,i,l}$.
\end{enumerate}
\end{theorem}

\begin{proof}
  Consider an element $f \in \MC(V)$. Define the positive
  $\Gamma$-invariant kernel $K \in \MC(V \times V)$ using invariant
  integration
\begin{equation*}
K(x,y) = \int_{\Gamma} f(ux) \overline{f(uy)}.
\end{equation*}
The eigenspaces of $T_K$ of nonzero eigenvalues are
finite-dimensional. They are invariant because for an eigenfunction $g
\in \MC(V)$ to an eigenvalue $\lambda$ we have
\begin{eqnarray*}
T_K(ug)(x) & = & \int_V K(x,y) g(u^{-1}y) d\mu(y) \\
& = & \int_V K(x,uy) g(y) d\mu(y)\\
& = & \int_V K(u^{-1}x, y) g(y) d\mu(y)\\
& = & \lambda g(u^{-1}x)\\
& = & \lambda ug(x).
\end{eqnarray*}
Now the statement follows: One breaks the invariant finite-dimensional
eigenspaces into irreducible subspaces using Maschke's theorem. The
orthonormal system one constructs using Gram-Schmidt
orthonormalization. The orthogonality relation follow from
Lemma~\ref{lem:orthogonalityrelation}. Completeness follows form the
spectral decomposition of positive kernels.
\end{proof}

\subsection{Bochner's characterization}
\label{ssec:bochner}

The complete orthonormal system $e_{k,i,l}$ of the Peter-Weyl theorem
is very useful to characterize $\Gamma$-invariant, positive
kernels. This is the contents of the following theorem by Bochner
\cite{Bochner}.

\begin{theorem}
\label{thm:bochner}
  Let $e_{k,i,l}$ be a complete orthonormal system for $\MC(V)$ as in
  Theorem~\ref{thm:peterweyl}. Every $\Gamma$-invariant, positive
  kernel $K \in \MC(V \times V)$ can be written as
\begin{equation}
\label{eq:bochnerrep}
  K(x,y) = \sum_{k = 0, 1, \ldots} \sum_{i,j = 1, 2 \ldots, m_k} f_{k,ij} \sum_{l = 1}^{h_k} e_{k,i,l}(x) \overline{e_{k,j,l}(y)},
\end{equation}
or more economically as
\begin{equation}
\label{eq:bochnerrepeco}
K(x,y) = \sum_{k = 0, 1, \ldots} \langle F_k, Z_k^{(x,y)} \rangle,
\end{equation}
with $(F_k)_{ij} = f_{k,ij}$ and $(Z_k^{(x,y)})_{ij} = \sum_{l =
  1}^{h_k} e_{k,i,l}(x) \overline{e_{k,j,l}(y)}$. Here $F_k$ is
Hermitian and positive. The series converges absolutely and uniformly.
\end{theorem}

\begin{proof}
  It is clear the every kernel of the form \eqref{eq:bochnerrep} or
  \eqref{eq:bochnerrepeco} is a positive $\Gamma$-invariant kernel.
  Generally, a kernel $K \in \MC(V \times V)$ can be written as a
  series in the basis $(x,y) \mapsto e_{k,i,l}(x)
  \overline{e_{k',i',l'}(y)}$. Let $f_{k,i,l;k',i',l'}$ be the
  corresponding coefficient. We shall show that $f_{k,i,l;k',i',l'} =
  0$ if $k \neq k'$ or $l \neq l'$ and that $f_{k,i,l;k,i',l}$ does
  not depend on $l$, so that we can set $f_{k,i,l;k,i',l} =
  f_{k,ii'}$. For this we consider the, possibly degenerate,
  $\Gamma$-invariant inner product on $\MC(V)$ defined by $K$:
\begin{equation*}
\label{eq:innerproductK}
  (f,g)_K = \int_V \int_V f(x) K(x,y) \overline{g(y)} d\mu(x)d\mu(y).
\end{equation*}
For this inner product, the orthogonality relations of
Lemma~\ref{lem:orthogonalityrelation} apply, which proves the claim
about the coefficients $f_{k,i,l;k',i',l'}$. Furthermore, the inner
product \eqref{eq:innerproductK} defines an inner product on the space
spanned by the vectors
\begin{equation*}
\varphi_{k,i} = (e_{k,i,1}, \ldots, e_{k,i,h_k}), \quad i = 1, \ldots, m_k,
\end{equation*}
by
\begin{equation*}
(\varphi_{k,i}, \varphi_{k,i'}) = \left(\sum_{l=1}^{h_k} e_{k,i,l}, \sum_{l=1}^{h_k} e_{k,i',l}\right)_K.
\end{equation*}
Hence, $F_k = (f_{k,ii'})$ is positive.
\end{proof}

Bochner's characterization is \defi{the} tool for exploiting symmetry
in semidefinite programs. If one wants to optimize over all
$\Gamma$-invariant kernels, then one only has to optimize over
positive $F_k$'s which have size $m_k \times m_k$. The computation of
$Z_k^{(x,y)}$ is part of the preprocessing. It can be done by hand, as in the
next two sections, or by computer if the group is finite and has
moderate size. See the article Babai, R\'onyai \cite{BabaiRonyai} for
algorithmic aspects.

\section{Boolean harmonics}
\label{sec:boolean}

In this section we specialize the Peter-Weyl theorem and the
characterization of Bochner to the $n$-dimensional Hamming cube
$\{0,1\}^n$. We apply these results to calculate the theta function
for distance graph on $\{0,1\}^n$. Then we give a complete proof of
the specialization using Fourier analysis of Boolean functions.

\subsection{Specializations}
\label{ssec:hammingspecialization}

Recall that the space of continuous functions $\MC(\{0,1\}^n)$, also
known as complex vectors indexed by $\{0,1\}^n$, comes with an inner
product
\begin{equation*} 
(f,g) = \frac{1}{2^n} \sum_{x \in \{0,1\}^n} f(x) \overline{g(x)},
\end{equation*}
which is invariant under the action of the automorphism group. The
group has order $2^n n!$ and it is generated by permutations and
switches $0 \leftrightarrow 1$. We make extensive use of the functions
$\chi_y(x) = (-1)^{y \cdot x}$, which are called
\defi{characters}. The character $\chi_y$ evaluates to $1$ if the
Hamming distance between $x$ and $y$ is an even integer and to $-1$ if
it is an odd integer.

\medskip

The Peter-Weyl theorem specializes as follows.

\begin{theorem}
\label{thm:booleanpeterweyl}
  The space $\MC(\{0,1\}^n)$ decomposes orthogonally into irreducible
  subspaces $H_k$:
\begin{equation*}
\MC(\{0,1\}^n) = H_0 \perp H_1 \perp \ldots \perp H_n.
\end{equation*}
Here $H_k$ denotes the subspace spanned by the functions $\chi_y(x) =
(-1)^{y \cdot x}$, with Hamming norm $\|y\| = k$. It is an irreducible
subspace of dimension $h_k = \binom{n}{k}$.
\end{theorem}

In the decomposition all irreducible subspaces have multiplicity
$1$. We have $H_k = H_{k,1}$ in the notation of
Theorem~\ref{thm:peterweyl}.  Hence, Bochner's characterization
(Theorem~\ref{thm:bochner}) of the $\Aut(\{0,1\}^n)$-invariant,
positive kernels only involves positive semidefinite matrices of size
$1 \times 1$, i.e.\ nonnegative scalars. The automorphism group acts
\defi{distance transitively} on pairs: For every two pairs of points
$(x,y)$ and $(x',y')$ having the same Hamming distance $\delta(x,y) =
\delta(x',y')$ there is a group element $u \in \Aut(\{0,1\}^n)$ which
transforms the pairs into each other $(ux,uy) = (x',y')$. So
$\Aut(\{0,1\}^n)$-invariant kernels are actually functions depending
only on the distance $\delta$.

\begin{theorem}
\label{thm:booleanbochner}
  Every $\Aut(\{0,1\}^n)$-invariant, positive kernel $K \in
  \MC(\{0,1\}^n \times \{0,1\}^n)$ is of the form
\begin{equation*}
K(x,y) = \sum_{k = 0}^n f_k K^n_k(\delta(x,y)), \quad f_0, \ldots, f_n \geq 0,
\end{equation*}
where $K^n_k$ denotes a the Krawtchouk polynomial of degree $k$.
\end{theorem}

\defi{Krawtchouk polynomials} $K^n_k(t)$ are orthogonal polynomials
for the inner product
\begin{equation*}
(P,Q) = \sum_{t = 0}^n \binom{n}{t} P(t) Q(t),
\end{equation*}
which are normalized by $K^n_k(0) = \binom{n}{k}$.  The first few
Krawtchouk polynomials are
\begin{equation*}
\begin{split}
& K^n_0(t) = 1, \\
& K^n_1(t) = -2t + n, \\
& K^n_2(t) = 2t^2 - 2nt + \binom{n}{2},
\end{split}
\end{equation*}
and in general 
\begin{equation*}
K^n_k(t) = \sum_{i = 0}^k \binom{t}{i} \binom{n-t}{k-i} (-1)^i.
\end{equation*}
See the books Szeg\"o \cite{Szego} or MacWilliams, Sloane
\cite{MacWilliamsSloane} for more information on Krawtchouk
polynomials.

\subsection{Linear programming bound for binary codes}
\label{ssec:hamminglp}

Now the calculation of the theta function on the Hamming cube reduces
to Delsarte's linear programming bound which was introduced by
Delsarte \cite{Delsarte}. The connection between Lov\'asz' theta
function and Delsarte's linear programming bound was first observed by
McEliece, Rodemich, Rumsey Jr.\ \cite{McElieceRodemichRumsey} and
independently by Schrijver \cite{Schrijver1}.

\begin{equation*}
\begin{split}
  \vartheta'(\{0,1\}^n, \{1,\ldots,d-1\}) = \min\Big\{ \lambda \;\; :
\;\; &
\text{$K \in \MC(\{0,1\}^n \times \{0,1\}^n)$}\\
& \text{$\quad$ positive semidefinite,}\\
& \text{$\quad$ $\Aut(\{0,1\}^n)$-invariant},\\
& \text{$K(x,x) = \lambda - 1$ for all $x \in \{0,1\}^n$,}\\
& \text{$K(x,y) \leq -1$ if $\delta(x,y) \in \{d, \ldots, n\}$}\Big\}.
\end{split}
\end{equation*}
Now we represent the positive semidefinite matrix $K \in \MC(\{0,1\}^n
\times \{0,1\}^n)$ which is $\Aut(\{0,1\}^n)$-invariant by $K(x,y) =
\sum_{k=0}^n f_k K^n_k(\delta(x,y))$ with $f_0, \ldots, f_n \geq 0$,
using the characterization of Bochner, and get the following linear
program in the variables $f_0, \ldots, f_n$. 
\begin{equation*}
\begin{split}
 \min\Big\{ 1 + \sum_{k = 0}^n f_k K^n_k(0)  \;\; : \;\; &
\text{$f_0, \ldots, f_n \geq 0$,}\\[-0.25cm]
& \text{$\sum_{k = 0}^n f_k K^n_k(t) \leq -1$ if $t \in \{d, \ldots, n\}$}\Big\}.
\end{split}
\end{equation*}

\begin{exercise}
  Calculate $\vartheta'(\{0,1\}^n,\{1, \ldots, d-1\})$ for, say, all
  $n = 1,\ldots, 128$ and all $d$.
\end{exercise}

\subsection{Fourier analysis on the Hamming cube}
\label{ssec:hammingfourier}

To show Theorem~\ref{thm:booleanpeterweyl} and
Theorem~\ref{thm:booleanbochner} we make use of discrete Fourier
analysis of Boolean functions. This theory has recently been
successfully developed and applied to problems concerning threshold
phenomena and influence in combinatorics, computer science, economics,
and political science. E.g.\ H\aa{}stad's result on the hardness of
the stability number, we refered earlier to, uses Fourier analysis of
Boolean functions. Kalai and Safra \cite{KalaiSafra} survey these
developments. One can find further information on Boolean harmonics
e.g.\ in the book Terras \cite{Terras}, in the paper Dunkl
\cite{Dunkl}, or in Delsarte's thesis \cite{Delsarte}.

\medskip

A complex-valued function on the $n$-dimensional Hamming cube $f :
\{0,1\}^n \to \C$ has the \defi{Fourier expansion}
\begin{equation*}
f(x) = \sum_{y \in \{0,1\}^n} \widehat{f}(y) \chi_y(x),
\end{equation*}
with the characters $\chi_y(x) = (-1)^{y \cdot x}$ and the
\defi{Fourier coefficients} $\widehat{f}(y)$. The characters $\chi_y$,
$y \in \{0,1\}^n$, which sometimes are called \defi{Walsh functions},
form an orthonormal basis of the space $\MC(\{0,1\}^n)$. The
one-dimensional subspace spanned by each $\chi_y$ is invariant under
the abelian group $(\{0,1\}^n, +)$ of addition modulo $2$, which
corresponds to the switching $0 \leftrightarrow 1$:
\begin{equation*}
u\chi_y(x) = \chi_y(x - u) = (-1)^{y \cdot (x-u)} = (-1)^{y \cdot (-u)} \chi_y(x).
\end{equation*}
Under the complete automorphism group of the Hamming cube, i.e.\ when
also permutations are allowed, then the one-dimensional subspaces are
grouped together according the Hamming norm of $y$. The
$\binom{n}{k}$-dimensional spaces
\begin{equation*}
H_k = \lin\{\chi_y : \|y\| = k\}, \quad k = 0, \ldots, n,
\end{equation*}
are $\Aut(\{0,1\}^n)$-invariant. We shall show that they are
irreducible and compute the corresponding basis function $Z_k$ for
Bochner's characterization.

\medskip

For this the notion of zonal spherical function is helpful. A function
$f : \{0,1\}^n \to \C$ is called a \defi{zonal spherical function with
  pole $e \in \{0,1\}^n$} if $f = uf$ for all $u \in \Aut(\{0,1\}^n)$
which stabilize the point $e$, so that $ue = e$.

\begin{lemma}
\label{lem:zonal}
  Let $U \subseteq \MC(\{0,1\}^n)$ be an invariant subspace and let $e
  \in \{0,1\}^n$. If the subspace of zonal spherical functions with
  pole $e$ which lie in $U$ has dimension exactly one, then $U$ is
  irreducible.
\end{lemma}

\begin{proof}
  Suppose $U$ is not irreducible and $U$ decomposes into nontrivial
  invariant subspaces $U = V \perp V^{\perp}$. Let $e_1, \ldots, e_n$
  be an orthonormal system of $V$ and let $f_1, \ldots, f_m$ be one of
  $V^{\perp}$. Then the functions $z_V(x) = \sum_{i = 1}^n e_i(e)
  \overline{e_i(x)}$, and $z_{V^{\perp}}(x) = \sum_{i = 1}^n f_i(e)
  \overline{f_i(x)}$ are two linearly independent zonal spherical
  functions with pole $e$.
\end{proof}

We apply this lemma to $H_k$. Let $f \in H_k$ be a zonal spherical
function with pole $0^n = (0, \ldots, 0)$. The group stabilizing $0^n$
is the group of all permutations. For every permutation $u$ we have
$uf = f$ and in terms of the Fourier expansion:
\begin{equation*}
\sum_{y \in \{0,1\}^n, \|y\| = k} \widehat{f}(y) \chi_{uy}(x) = 
\sum_{y \in \{0,1\}^n, \|y\| = k} \widehat{f}(y) \chi_{y}(x).
\end{equation*}
So the Fourier coefficients $\widehat{f}(y)$ all have to coincide.

\begin{exercise}
Show that the $H_k$'s are pairwise not equivalent.
\end{exercise}

Let us calculate the basis function $Z_k(x,x')$ for Bochner's
characterization (and in this way the zonal spherical functions of
$H_k$ because a zonal spherical function with pole $e$ is a multiple
of $Z_k(e,\cdot)$ as one sees from the proof of
Lemma~\ref{lem:zonal}). We have
\begin{equation*}
Z_k(x,x') = \sum_{\|y\| = k} \chi_u(x) \overline{\chi_u(x')}.
\end{equation*}
Since $Z_k$ is invariant under $\Aut(\{0,1\}^n)$ it only depends on
the Hamming distance $t = \delta(x,x')$ and so we can assume that $x =
0$, and $x' = 1^t0^{n-t}$. A straightforward combinatorial calculation
yields
\begin{equation*}
  Z_k(x,x') = Z_k(0,1^t0^{n-t}) = \sum_{i = 0}^k (-1)^i \binom{t}{i} \binom{n-t}{k-i} = K^n_k(t).
\end{equation*}
Because of Lemma~\ref{lem:orthogonalityrelation} we have the
orthogonality relation
\begin{equation*}
\sum_{t = 0}^n \binom{n}{t} K^n_k(t) K^n_{k'}(t) = 0, \quad \text{if $k \neq k'$.}
\end{equation*}

\section{Spherical harmonics}
\label{sec:spherical}

In this section we replace the $n$-dimensional Hamming cube by the
unit sphere $S^{n-1}$ in $n$-dimensional Euclidean space. First we
specialize the Peter-Weyl theorem and Bochner's characterization.  We
apply these results to calculate the theta function for distance graph
on the sphere. This gives tight bounds for the kissing number in
dimension $8$ and $24$. Then we give a complete proof of the
specialization using the theory of spherical harmonics.

\medskip

Spherical harmonics have been extensively studied due to their
importance in mathematics, physics and engineering. Comprehensive
information is available. Here we follow to some extend the book
\cite[Chapter 9]{AndrewsAskeyRoy} by Andrews, Askey and Roy, and the
paper \cite{CoifmanWeiss} by Coifman and Weiss which links spherical
harmonics with the representation theory of compact topologic
groups. A lot of information is contained in the book
\cite{VilenkinKlimyk} by Vilenkin and Klimyk. Also the book
\cite{Mueller} by M\"uller is very useful.

\medskip

Spherical harmonics serve as a full-fledged example of geometric
analysis: See the books \cite{Berger} by Berger and \cite{Helgason} by
Helgason. The paper by Stanton \cite{Stanton} emphasizes the
similarities between spherical harmonics and Boolean harmonics.

\subsection{Specializations}
\label{ssec:spherespecialization}

Recall that the space of continuous functions $\MC(S^{n-1})$ comes
with an inner product
\begin{equation*} 
(f,g) = \int_{S^{n-1}} f(x) \overline{g(x)} d\omega(x),
\end{equation*}
which is invariant under the action of the orthogonal group $\On$,
the automorphism group of the unit sphere.

\medskip

The Peter-Weyl theorem specializes as follows.

\begin{theorem}
\label{thm:harmonics}
  The space of complex-valued continuous functions on the unit sphere
  decomposes orthogonally into pairwise $\On$-irreducible subspaces as
  follows:
\begin{equation*}
\MC(S^{n-1}) = H_0 \perp H_1 \perp H_2 \perp \ldots
\end{equation*}
Here $H_k$ denotes the space of homogeneous polynomial functions of
degree $k$ which vanish under the Laplace operator $\Delta =
\frac{\partial^2}{\partial x_1^2} + \cdots + \frac{\partial^2}{\partial
  x_n^2}$. The space $H_k$ has dimension $h_k = \binom{n+k-1}{n-1} -
\binom{n+k-3}{n-1}$.
\end{theorem}

Since in the decomposition all irreducible subspaces have multiplicity
$1$, i.e\ we have $H_k = H_{k,1}$ in the notation of
Theorem~\ref{thm:peterweyl}, Bochner's characterization of the
$\On$-invariant, positive kernels again only involves nonnegative
scalars. The orthogonal group acts distance transitively on pairs of
points on the sphere. So $\On$-invariant kernels are actually
functions depending only on the spherical distance $d(x,y) = \arccos x
\cdot y$. It turns out that they are univariate polynomial functions
in the inner products $t = x \cdot y$. Schoenberg in \cite{Schoenberg}
was the first who gave this characterization.

\begin{theorem}
\label{thm:schoenberg}
  Every positive, $\On$-invariant kernel $K \in \MC(S^{n-1} \times
  S^{n-1})$ can be written as
\begin{equation*}
  K(x,y) = \sum_{k = 0}^{\infty} f_k P^{(\alpha,\alpha)}_k(x \cdot y), \quad f_0, f_1, \ldots \geq 0.
\end{equation*}
The right hand side converges absolutely and uniformly.  Here
$P^{(\alpha,\alpha)}_k(t)$ is the normalized Jacobi polynomial of
degree $k$ with parameters $(\alpha,\alpha)$, where $\alpha =
(n-3)/2$.
\end{theorem}

The \defi{Jacobi polynomials} with parameters $(\alpha, \beta)$ are
orthogonal polynomials for the measure $(1-t)^\alpha(1+t)^\beta dt$ on
the interval $[-1,1]$. We denote by $P^{(\alpha,\beta)}_k$ the
\defi{normalized Jacobi polynomial} of degree $k$ with normalization
$P^{(\alpha,\beta)}_k(1) = 1$. The first few Jacobi polynomials are
\begin{equation*}
\begin{split}
& P^{(\alpha,\alpha)}_0(t) = 1, \\
& P^{(\alpha,\alpha)}_1(t) = t, \\
& P^{(\alpha,\alpha)}_2(t) = \frac{n}{n-1}t^2 - \frac{1}{n-1}.
\end{split}
\end{equation*}
For more information on Jacobi polynomials we refer to the books
\cite{AbramowitzStegun} by Abramowitz, Stegun, \cite{Szego} by
Szeg\"o, \cite{AndrewsAskeyRoy} by Andrews, Askey and Roy, and
\cite{Lebedev} by Lebedev.

\subsection{Linear programming bound for the unit sphere}
\label{ssec:spherelp}

Before we give a proof of these two theorems, we demonstrate how to
apply them to calculate the theta function for distance graphs on the
unit sphere. Again the original semidefinite program degenerates to a
linear program because $\MC(S^{n-1})$ decomposes multiplicity-free
into $\Aut(S^{n-1})$-irreducible subspaces. This linear program was
studied by Delsarte, Goethals and Seidel in
\cite{DelsarteGoethalsSeidel} and by Kabatiansky and Levenshtein in
\cite{KabatianskyLevenshtein}, but see also Sloane \cite[Chapter
9]{ConwaySloane}.

\medskip

A case for which the calculation of the theta function is particularly
simple is the graph $G(S^{n-1}, (0,\pi/2))$. Every vertex is adjacent
to a pointed half-sphere. Despite its simplicity this case already
carries all the features of the more complicated kissing number cases
$G(S^7, (0,\pi/3))$ and $G(S^{23}, (0,\pi/3))$.

\medskip

We show that the vertices of a regular cross polytope (the $2n$ points
$\pm e_i$, with $i = 1, \ldots, n$) form an optimal spherical code. We
say that $N$ points on the unit sphere $S^{n-1}$ form an \defi{optimal
  spherical code} if they maximize the minimal distance among all
$N$-point configuration on $S^{n-1}$.

\medskip

Using Theorem~\ref{thm:schoenberg} the theta function
$\vartheta'(G(S^{n-1}, (0,\pi/2)))$ simplifies to
\begin{equation*}
\begin{split}
 \inf\Big\{ 1 + \sum_{k = 0}^{\infty} f_k P^{(\alpha,\alpha)}_k(0)  \;\; : \;\; &
\text{$f_0, f_1, \ldots \geq 0$,}\\[-0.25cm]
& \text{$\sum_{k = 0}^{\infty} f_k P^{(\alpha,\alpha)}_k(\cos \theta) \leq -1$ if $\theta \in [\pi/2,\pi]$}\Big\}.
\end{split}
\end{equation*}

Consider a polynomial $F(t) = \sum_{k=0}^d f_k
P^{(\alpha,\alpha)}_k(t)$ with nonnegative coefficients $f_k$. If
$F(t) \leq -1$ for all $t \in [-1,0]$, then
\begin{equation*}
F(0) + 1 \geq \vartheta'(G(S^{n-1}, (0,\pi/2))) \geq \alpha(G(S^{n-1}, (0,\pi/2))).
\end{equation*}

We give an explicit $F$ which proves the sharp bound of
$2n$. Generally, if a polynomial proves a sharp bound, then all
inequalities in Theorem~\ref{thm:upperbound} have to be equalities. So
we have $F(t) = -1$ for all $t$ which occur as inner product.

\medskip

We make the Ansatz
\begin{equation*}
F(t) = f_0 + f_1 t + f_2 \left(\frac{n}{n-1} t^2 - \frac{1}{n-1}\right)
\end{equation*}
and in order that the bound is sharp we have to have
\begin{eqnarray*}
1 + F(1) & = & 2n,\\
F(0) & = & -1,\\
F(-1) & = & -1.
\end{eqnarray*}
So we consider the system of linear equations
\begin{equation*}
\begin{pmatrix}
1 & 1 & 1\\
1 & 0 & -\frac{1}{n-1}\\
1 & -1 & 1
\end{pmatrix}
\begin{pmatrix}
f_0\\f_1\\f_2
\end{pmatrix}
= 
\begin{pmatrix}
2n-1\\
-1\\
-1\\
\end{pmatrix},
\end{equation*}
which happens to have the nonnegative solution $f_0 = 0$, $f_1 = n$,
$f_2 = n-1$.

\begin{exercise}
  Why does this show that the vertices of a regular cross polytope
  form an optimal spherical code?
\end{exercise}

\begin{exercise}
Show that $\tau_7 = \vartheta'(G(S^7, (0,\pi/3))) = 240$.
\end{exercise}

\begin{exercise}
Show that $\vartheta'(G(S^{24}, (0,\pi/3))) = 196560$.
\end{exercise}

\begin{exercise}
  Show that the vertices of the icosahedron form an optimal spherical
  code.
\end{exercise}

For more on optimal spherical codes see \cite{BannaiSloane},
\cite{Levenshtein}, \cite{Cuypers}, \cite{CohnKumar2},
\cite{BallingerBlekhermanCohnGiansiracusaKellySchuermann},
\cite{BachocVallentin3}. Recently, Wang \cite{Wang} gave a list of
conjectural optimal spherical codes up to dimension $8$ and up to
$27$ points. It is a challenge to find proofs of their optimality.

\subsection{Harmonic polynomials}
\label{ssec:harmonic polynomials}

To establish the specialization of the Peter-Weyl theorem for the case
of the unit sphere we consider harmonic polynomials.

\medskip

For the vector space of complex polynomial functions $f : \C^n \to \C$
which are homogeneous and are of degree $d$ we write
$\Pol_{=d}(\C^n)$. Each $f \in \Pol_{=d}(\C^n)$ satisfies the equation
$f(\alpha x) = \alpha^d f(x)$ for all $\alpha \in \C$ and $x \in
\C^n$. The dimension of $\Pol_{=d}(\C^n)$ is $\binom{n+d-1}{n-1}$. We
define the polynomial $\omega \in \Pol_{=2}(\C^n)$ by $ \omega(x_1,
\ldots, x_n) = x_1^2 + \cdots + x_n^2$.

Furthermore, we introduce two differential operators: The
\defi{Nabla operator}
\begin{equation*}
\nabla = (\frac{\partial}{\partial x_1}, \ldots,
\frac{\partial}{\partial x_n})^t
\end{equation*}
and the \defi{Laplace operator}
\begin{equation*}
\Delta = \nabla^t \nabla = \frac{\partial^2}{\partial x_1^2} + \cdots
+ \frac{\partial^2}{\partial x_n^2} = \omega(\nabla),
\end{equation*}
which maps $\Pol_{=d}(\C^n)$ to $\Pol_{=d-2}(\C^n)$. It will be
convenient to have the following inner product on $\Pol_{=d}(\C^n)$:
\begin{equation*}
\langle f, g \rangle = \frac{1}{d!} f(\nabla) \overline{g},
\end{equation*}
for which the basis of monomials is an orthogonal basis with
\begin{equation*}
\langle x_1^{m_1} \cdots x_n^{m_n}, x_1^{m_1} \cdots x_n^{m_n} \rangle
= \frac{m_1! \cdots m_n!}{d!}.
\end{equation*}
This inner product might look mysterious, but it turns out to be very
natural if one considers polynomials as symmetric tensors as we point
out in the appendix.

\medskip

It is a simple and helpful fact that multiplication with $\omega$ is
the adjoint to the Laplace operator:

\begin{lemma}
\label{adjoint}
For all $f \in \Pol_{=d-2}(\C^n)$ and for all $g \in \Pol_{=d}(\C^n)$
we have
\begin{equation*}
\langle \omega f, g \rangle = \langle f, \Delta g \rangle.
\end{equation*}
\end{lemma}

\begin{proof}
  This follows immediately from the obvious fact that the evaluation map
$f \mapsto f(\nabla)$ is an algebra isomorphism between the algebra
of polynomials $\C[x_1, \ldots, x_n]$ and the algebra of differential
operators $\C[\frac{\partial}{\partial x_1}, \ldots,
\frac{\partial}{\partial x_n}]$. In particular, this implies $(\omega
f)(\nabla) = f(\nabla) \omega(\nabla)$, and hence $ \langle \omega f,
g \rangle = \frac{1}{d!} f(\nabla) \overline{(\Delta g)} = \langle f, \Delta g \rangle$.
\end{proof}

\medskip

The kernel of the Laplace operator is the space of \defi{harmonic
polynomials} which we denote by
\begin{equation*}
\Harm_d = \{f \in \Pol_{=d}(\C^n) : \Delta f = 0\}.
\end{equation*}

\medskip

\begin{theorem}
The space $\Pol_{=d}(\C^n)$ decomposes orthogonally into harmonic
polynomials as follows:
\begin{equation*}
\Pol_{=d}(\C^n) = \Harm_d \perp \omega \Harm_{d-2} \perp \omega^2 \Harm_{d-4} \perp \ldots.
\end{equation*}
The dimension of $\Harm_k$, which we denote by $h_k$, is equal to
$\binom{n+k-1}{n-1} - \binom{n+k-3}{n-1}$.
\end{theorem}

\begin{proof}
We first show that $\Delta$ maps $\Pol_{=d}(\C^n)$ onto
$\Pol_{=d-2}(\C^n)$: By Lemma~\ref{adjoint} we have for $f \in
(\Delta(\Pol_{=d-2}(\C^n)))^\perp$ and for $g \in \Pol_{=d}(\C^n)$
\[
\langle\omega f, g\rangle = \langle f, \Delta g\rangle = 0,
\]
hence $\omega f$, and so $f$, must be $0$. The dimension formula for
linear maps implies the desired formula for $h_k$.

With similar arguments one shows that
\begin{equation*}
\label{first decomposition}
\Pol_{=d}(\C^n) = \Harm_d \perp \omega \Pol_{=d-2}(\C^n).
\end{equation*}

Before we can inductively derive the orthogonal decomposition as
stated in the theorem we only have to verify that $\omega
\Harm_{d-2}$ is orthogonal to $\omega^2 \Pol_{=d-4}(\C^n)$. The
following identity will be helpful: for all $p \in \Pol_{=d}(\C^n)$,
we have $\Delta \omega p = 2(n + 2d) p + \omega \Delta p$ (in the
calculation, notice that $\sum_{i=1}^n x_i \frac{\partial}{\partial
x_i} p = d p$). Then,
\begin{eqnarray*}
\langle \omega f, \omega^2 g\rangle
& = & \langle f, \Delta \omega^2 g\rangle\\
& = & \langle f, 2(n+2(d-2)) \omega g + \omega \Delta \omega g \rangle\\
& = & \langle f, \omega(2(n+2(d-2)) g + \Delta \omega g) \rangle\\
& = & 0,
\end{eqnarray*}
for all $f \in \Harm_{d-2}$ and $g \in \Pol_{=d-4}(\C^n)$.
\end{proof}

By $\Pol_{\leq d}(\Sn)$ we denote the vector space of complex-valued
polynomial functions on $\Sn$ which are of degree at most $d$. This is
a subspace of $\MC(S^{n-1})$ so it is equipped with the inner product
\[
\prodhaar{f}{g} = \int_{\Sn} f(x)\overline{g(x)} d\omega_n(x).
\]
Recall that the measures on $\Sn$ and $\Snm$ are related by
\begin{equation*}
\label{eq:measurerelation}
d\omega_n(x)=(1-x_1^2)^{(n-3)/2}dx_1
d\omega_{n-1}((x_2^2+\cdots+x_n^2)^{-1/2}(x_2, \ldots, x_n)).
\end{equation*}

The orthogonal group acts on the space of polynomial functions by
$uf(x) = f(u^{-1}x)$. The space of harmonic polynomials $\Harm_d$ is
invariant under this action. Hence, we have for all $f \in \Harm_d$
and all $u \in \On$ that $uf \in \Harm_d$. This follows from the fact
that $\Delta(uf) = u(\Delta f)$. Another more conceptual way to see
that $\Harm_d$ is an $\On$-invariant space uses the fact that the
inner product $\langle \cdot, \cdot \rangle$ is invariant under the
action of $\On$. We will give a proof of this in the appendix where
the connection between symmetric tensors and polynomials will also
make the inner product $\langle \cdot, \cdot \rangle$ more
transparent. It is obvious that the summand $\omega \Pol_{=d-2}(\C^n)$
in the decomposition \eqref{first decomposition} is an $\On$-invariant
space. Then, its orthogonal complement, $\Harm_d$, is an
$\On$-invariant space, too.

\medskip

The space of \textit{spherical harmonics}
\begin{equation*}
H_k = \{f_{|\Sn} : f \in \Harm_k\}
\end{equation*}
is also an $\On$-invariant space. The spaces $\Harm_k$ and $H_k$ are
equivalent as $\On$-invariant spaces: the restriction map $\varphi(f)
= f_{|\Sn}$ is a bijective linear map $\varphi : \Harm_k \to H_k$ with
$\varphi(uf) = u\varphi(f)$ for all $f \in \Harm_k$ and all $u \in
\On$. Note that the inverse map of $\varphi$ is given by
$\varphi^{-1}(f)(x) = \|x\|^k f\big(\frac{x}{\|x\|}\big)$.

\medskip

The following construction will turn out to be important: For an
orthonormal system $e_1, \ldots, e_{h_k}$ of $H_k$ we define the
function
\begin{equation*}
z_k : \Sn \times \Sn \to \C \quad \mbox{by} \quad z_k(x, y) = \sum_{i
= 1}^{h_k} e_i(x) \overline{e_i}(y).
\end{equation*}

\begin{proposition}
\label{z_k}
The following properties hold for the function $z_k$:

(a) $z_k$ does not dependent on the choice of the orthonormal system of
$H_k$.

(b) For all $u \in \On$ and for all $x, y \in \Sn$ we have $z_k(ux,
uy) = z_k(x, y)$.

\end{proposition}

\begin{proof}
(a) Let $e'_1, \ldots, e'_{h_k}$ be a second orthonormal system of
$H_k$. Let $u$ be the orthogonal matrix $u = (u_{ij}) \in
\Orth(\R^{h_k})$ with $ue_i = e'_i$, $i = 1, \ldots, h_k$. Then,
\begin{eqnarray*}
\sum_{i = 1}^{h_k} e'_i(x) \overline{e'_i(y)} & = &
\sum_{i = 1}^{h_k} (u e_i)(x) \overline{(u e_i)(y)}\\
& = & \sum_{i = 1}^{h_k} \sum_{j = 1}^{h_k} u_{ji} e_j(x) u_{ji} \overline{e_j(y)}\\
& = & \sum_{j = 1}^{h_k} \Big(\sum_{i = 1}^{h_k} u_{ji} u_{ji}\Big) e_j(x) \overline{e_j(y)}\\
& = & \sum_{j = 1}^{h_k} e_j(x) \overline{e_j(y)}.
\end{eqnarray*}

(b) Since the inner product on $\Pol_{\leq d}(\Sn)$ is invariant under
$\On$, it follows that $u^{-1}e_i$, $i = 1, \ldots, h_k$ is an
orthonormal system of $H_k$. Now (b) is a consequence of (a).
\end{proof}

As in the case of Boolean harmonics we make use of the concept of
zonal spherical functions to prove the irreducibility of $H_k$. A
function $f : \Sn \to \C$ is called a \defi{zonal spherical function
  with pole} $e \in \Sn$ if $f = uf$ for all $u \in \Stab(\On, e) =
\{u \in \On : ue = e\}$. Property (b) of Proposition \ref{z_k} implies
that $x \mapsto z_k(e, x)$ is a zonal spherical function with pole
$e$. In direct analogy to Lemma~\ref{lem:zonal} we have:

\begin{lemma}
\label{zonal irred}
Let $U \subseteq \Pol_{\leq d}(S^{n-1})$ be an $\On$-invariant space
and let $e \in \Sn$ be a point on the sphere. If the dimension of the
space of zonal spherical functions with pole $e$ which lie in $U$ is
exactly one, then $U$ is an $\On$-irreducible space.
\end{lemma}

Now the specialization of the Peter-Weyl theorem given in
Theorem~\ref{thm:harmonics} follows form the following theorem.

\begin{theorem}
\label{second decomposition}
The space of complex-valued polynomials restricted to the sphere which
are of degree at most $d$ decomposes orthogonally into the spaces of
spherical harmonics:
\begin{equation*}
\Pol_{\leq d}(\Sn) = H_0 \perp H_1 \perp \ldots \perp H_d.
\end{equation*}
The $H_k$ are $\On$-irreducible spaces which are pairwise
inequivalent.
\end{theorem}

\begin{proof}
  Using Lemma~\ref{zonal irred} we show that the $H_k$'s are irreducible.
  Let $f \in H_k$ be a zonal spherical function with pole $e =
  (1,0,\ldots,0)^t$. We consider $f$ as a polynomial in
  $\Harm_k$. Then we can write
\begin{equation*}
f(x_1, \ldots, x_n) = \sum_{i = 0}^k x_1^{k-i}p_i(x_2,\ldots,x_n),
\end{equation*}
where $p_i$ is a homogeneous polynomial of degree $i$ which is
invariant under the action of the stabilizer group $\Stab(\On,
e)$. That means that $p_i$ is a radial function, a function which only
depends on the norm of $(x_2,\ldots,x_n)$. Because $p_i$ is
homogeneous of degree $i$ we have $p_i(x_2, \ldots, x_n) = c_i (x_2^2
+ \cdots + x_n^2)^{i/2}$ for some constant $c_i$. Hence, $p_i = 0$
whenever $i$ is an odd integer. So we can write
\begin{equation*}
f(x_1,\ldots,x_n) = \sum_{i = 0}^{k/2} c_i x_1^{k-2i}(x_2^2 + \cdots + x_n^2)^i.
\end{equation*}
Since $f$ is a harmonic polynomial we have (check this)
\begin{equation*}
0 = \Delta f = \sum_{i = 1}^{k/2} (\alpha_i c_i + \beta_i c_{i-1})x_1^{k-2i}(x_2^2 + \cdots + x_n^2)^{i-1},
\end{equation*}
where $\alpha_i = 2i(n+2i-3)$, $\beta_i = (k-2i+1)(k-2i+2)$, and so
$c_i = (-1)^i c_0 \frac{\beta_1 \cdots \beta_i}{\alpha_1 \ldots
  \alpha_i}$, for $i = 1, \ldots, k/2$, which shows that $f$ is
determined by $c_0$ only. Thus, the space of zonal spherical functions
is one-dimensional.

\medskip

If $n > 2$, then, by looking at the dimension $h_k$, it follows
immediately that $H_k$ are pairwise inequivalent and also the
orthogonality relation between them. The case $n = 2$ we leave as an
exercise (see below).
\end{proof}

\begin{exercise}
  Make the link between the spherical harmonics for the unit
  circle~$S^{1}$ and the classical Fourier series of $2\pi$-periodic
  functions. Show how the the functions $\cos kx$, $\sin kx$, with $k
  = 0, 1, 2, \ldots, $ appear as orthonormal system for the spaces
  $H_k$. Show that $H_k$ and $H_{k'}$ with $k \neq k'$ are not
  equivalent.
\end{exercise}

\subsection{Addition formula}
\label{ssec:addition formula}

To establish Bochner's characterization for the sphere we still have
to calculate $Z_k(x,y)$. For this we do not need to write down an
orthonormal system of $H_k$ explicitly. Studying the orthogonality
relations and their interplay with zonal spherical function suffices
here.

\begin{theorem}
\label{addtion formula theorem}
Let $e_1, \ldots, e_{h_k}$ be an orthonormal system of $H_k$. For $x,
y \in \Sn$ we have the addition formula
\begin{equation*}
P^{((n-3)/2,(n-3)/2)}_k(\prodeucl{x}{y}) = \alpha_k \sum_{l = 1}^{h_k} e_l(x) \overline{e_l(y)},
\end{equation*}
for a positive constant $\alpha_k$.
\end{theorem}

\begin{proof}
  Let $z_k$ and $z_k'$ be zonal spherical functions for $H_k$ and
  $H_{k'}$, with $k \neq k'$, both with pole $e = (1,0,\ldots,0)$. By
  the orthogonality relation between nonequivalent $\On$-irreducible
  subspaces it follows
\begin{equation*}
(z_k, z_{k'}) = \int_{S^{n-1}} z_k(x) \overline{z_{k'}(x)} d\omega(x) = 0
\end{equation*}
The functions $z_k$ and $z_{k'}$ are invariant under the stabilizer
group $\Stab(\On,e)$. So they only depend on the first coordinate. One
can use \eqref{eq:measurerelation} to express the integral above as
one involving only $x_1$:
\begin{eqnarray*}
\int_{-1}^1 z_k(x_1,\ldots, x_n) \overline{z_{k'}(x_1, \ldots, x_n)} (1-x_1^2)^{(n-3)/2}dx_1 = 0.
\end{eqnarray*}
Hence, the zonal spherical functions $z_k$ are multiples of Jacobi
polynomials: 
\begin{equation*}
z_k(x_1, \ldots, x_n) = \alpha_k
P^{((n-3/2),(n-3)/2)}_k(x_1),
\end{equation*}
and hence
\begin{equation*}
\sum_{l = 1}^{h_k} e_l(x)\overline{e_l(y)} = \alpha_k P^{((n-3/2),(n-3)/2)}_k(x \cdot y).
\end{equation*}
The constant $\alpha_k$ is positive because
$P^{((n-3/2),(n-3)/2)}_k(1) = 1$.
\end{proof}


%
%
%

\begin{appendix}

\section{Symmetric tensors and polynomials}
\label{sec:tensors}

In this section we write $V$ for $\C^n$. The symmetric group $\Perm_d$
of permutations $\sigma : \{1, \ldots, d\} \to \{1, \ldots, d\}$ acts
on $V^{\otimes d}$ by permuting positions:
\begin{equation*}
\sigma(v_1 \otimes \cdots \otimes v_d) = v_{\sigma^{-1}(1)} \otimes
\cdots \otimes v_{\sigma^{-1}(d)},
\end{equation*}
with $\sigma \in \Perm_d$. The subspace of \defi{symmetric $d$-tensors
  over $V$} is
\begin{equation*}
S_d(V^{\otimes d}) = \big\{ u \in V^{\otimes d} : \mbox{$\sigma u = u$
for all $\sigma \in \Perm_d$} \big\}.
\end{equation*}
The inner product on $V^{\otimes d}$ is given by
\begin{equation*}
\langle v_1 \otimes \cdots \otimes v_d, w_1 \otimes \cdots \otimes
w_d \rangle = \prod_{i=1}^d \prodeucl{v_i}{\overline{w_i}}.
\end{equation*}

Let $e_1, \ldots, e_n$ be the standard basis of $\C^n$. Define the
linear map 
\begin{equation*}
\pi : V^{\otimes d} \to \Pol_{=d}(\C^n),\;\; \mbox{by}\;\; e_{i_1}
\otimes \cdots \otimes e_{i_d} \mapsto x_{i_1} \cdots x_{i_d}.
\end{equation*}
For an $n$-tupel $(m_1, \ldots, m_n)$ of nonnegative integers with
$m_1 + \cdots + m_n = d$ define the \emph{symmetric monomial}
\begin{equation*}
s_{(m_1, \ldots, m_n)} = \sum_{(i_1, \ldots i_d)} e_{i_1} \otimes \cdots
\otimes e_{i_d},
\end{equation*}
where we sum over all $(i_1, \ldots, i_d)$ so that the number $j \in
\{1, \ldots, n\}$ occurs in $(i_1, \ldots, i_d)$ with multiplicity
$m_j$, e.g.
\begin{equation*}
s_{(1,2)} = e_1 \otimes e_2 \otimes e_2 + e_2 \otimes e_1 \otimes e_2
+ e_2 \otimes e_2 \otimes e_1.
\end{equation*}
The symmetric monomials form a basis of $S_d(V)$. We have
\begin{equation*}
\pi(s_{(m_1, \ldots, m_n)}) = \frac{d!}{m_1! \cdots m_n!}x_1^{m_1}
\cdots x_n^{m_n}.
\end{equation*}
Hence, the map $\pi$ restricted to symmetric $d$-tensors over $V$ is
an isometry due to the specific choice of the inner product $\langle
\cdot, \cdot \rangle$ in $\Pol_{=d}(\C^n)$.

The orthogonal group $\On$ acts on $V^{\otimes d}$ by $u(v_1 \otimes
\cdots \otimes v_d) = uv_1 \otimes \cdots \otimes uv_d$, with $u \in
\On$. Let us verify that $\pi$ commutes with the action of $\On$,
i.e.\ $\pi(us) = a\pi(s)$ for all $s \in S_d(V)$ and $u = (u_{i,j})
\in \On$. It suffices to consider the basis $s_{(m_1, \ldots,
m_n)}$. We have
\begin{equation*}
\begin{split}
\pi\big(u{s_{(m_1, \ldots, m_n)}}\big) & = \pi\Big(\sum_{(i_1, \ldots, i_d)}
ue_{i_1} \otimes \cdots \otimes ue_{i_d}\Big)\\ 
& = \pi\Big(\sum_{(i_1, \ldots, i_d)}\sum_{j_1, \ldots, j_d = 1}^n
u_{j_1,i_1} \cdots u_{j_d,i_d} e_{j_1} \otimes \cdots \otimes e_{j_d}
\Big)\\
& = \sum_{(i_1, \ldots, i_d)} \big(\sum_{j=1}^n u_{j,i_1}
x_j\big) \cdots \big(\sum_{j=1}^n u_{j,i_d}
x_j\big)\\
& = \frac{d!}{m_1! \cdots m_n!} \big(\sum_{j=1}^n u_{j,1} x_j\big)^{m_1}
\cdots \big(\sum_{j=1}^n u_{j,n} x_j\big)^{m_n}\\
& = u\pi(s_{(m_1, \ldots, m_n)}).
\end{split}
\end{equation*}

We finish by showing that the inner product on $\Pol_{=d}(\C^n)$ is
invariant under the action of the orthogonal group. For this note that
the inner product on $V^{\otimes d}$ is invariant under this group
action. From the previous considerations we have for all $u \in \On$
the chain of equalities:
\begin{equation*}
\begin{split}
\langle uf, ug \rangle & =
\langle \pi^{-1}(uf), \pi^{-1}(ug) \rangle\\
& = \langle u\pi^{-1}(f), u\pi^{-1}(g)\rangle\\
& = \langle \pi^{-1}(f), \pi^{-1}(g)\rangle \\
& = \langle f, g \rangle.
\end{split}
\end{equation*}

\section{Further reading}

In the last years many results were obtained for semidefinite programs
which are symmetric. This was done for a variety of problems and
applications: 
\begin{itemize}
\item[---] interior point algorithms (Kanno, Ohsaki, Murota, Katoh
  \cite{KannaOhsakiMurotaKatoh}, de Klerk, Pasechnik
  \cite{KlerkPasechnik1}, Murota, Kanno, Kojima, Kojima
  \cite{MurotaKannoKojimaKojima}),
\item[---] polynomial optimization (Gatermann and Parrilo
\cite{GatermannParrilo} and Jansson, Lasserre, Riener, Theobald
\cite{JanssonLasserreRienerTheobald}, Bosse \cite{Bosse}, Laurent
\cite{Laurent2}), 
\item[---] truss topology optimization (Bai, de Klerk,
Pasechnik, Sotirov \cite{BaiKlerkPasechnikSotirov}), 
\item[---] quadratic
assignment problem (de Klerk, Sotirov \cite{KlerkSotirov}), 
\item[---]
fast mixing Markov
chains on graphs (Boyd, Diaconis, Xiao \cite{BoydDiaconisXiao}, Boyd,
Diaconis, Parrilo, Xiao \cite{BoydDiaconisParriloXiao}), 
\item[---] graph
coloring (Gvozdenovi\'c, Laurent \cite{GvozdenovicLaurent1},
\cite{GvozdenovicLaurent2}, Gvozdenovi\'c \cite{Gvozdenovic}),
--- crossing numbers for complete binary graphs (de Klerk, Pasechnik,
Schrijver \cite{KlerkPasechnikSchrijver}), 
\item[---] travelling salesman problem (de Klerk, Pasechnik, Sotirov \cite{KlerkPasechnikSotirov}),
\item[---] coding theory (Schrijver \cite{Schrijver2}, Gijswijt,
  Schrijver, Tanaka \cite{GijswijtSchrijverTanaka}, de Klerk,
  Pasechnik \cite{KlerkPasechnik2}, de Klerk, Newman, Pasechnik,
  Sotirov \cite{KlerkNewmanPasechnikSotirov}, Laurent \cite{Laurent1},
  Bachoc \cite{Bachoc}, Vallentin \cite{Vallentin2}, Bachoc, Vallentin
  \cite{BachocVallentin4}),
\item[---] low distortion embeddings (Vallentin \cite{Vallentin1})
\item[---]
geometry (Bachoc,
Vallentin \cite{BachocVallentin1}, \cite{BachocVallentin2},
\cite{BachocVallentin3}, Bachoc, Nebe, de Oliveira Filho, Vallentin
\cite{BachocNebeOliveiraVallentin}).
\end{itemize}

\end{appendix}

\end{document}